\documentclass[a4paper]{amsart}

\usepackage[latin1]{inputenc}
\usepackage{amsmath, amsthm, amssymb}

\usepackage[colorlinks]{hyperref}

\newtheorem{problem}{Problem}
\newtheorem{theorem}{Theorem}
\newtheorem{lemma}{Lemma}
\newtheorem{corollary}{Corollary}

\newcommand{\C}{\mathbb{C}}
\newcommand{\R}{\mathbb{R}}
\newcommand{\T}{\mathbb{T}}
\newcommand{\Z}{\mathbb{Z}}
\newcommand{\Zn}{\Z_n}
\newcommand{\dZn}{\widehat{\Z}_n}

\newcommand{\Kcal}{\mathcal{K}}
\newcommand{\Pcal}{\mathcal{P}}

\newcommand{\defi}[1]{\textit{#1}}

\newcommand{\tp}{{\sf T}}

\DeclareMathOperator{\so}{SO}
\DeclareMathOperator{\vol}{vol}
\DeclareMathOperator{\motion}{M}
\DeclareMathOperator{\cayley}{Cayley}

\begin{document}

\title{Mathematical optimization for packing problems}

\author{Fernando M\'ario de Oliveira Filho}
\address{F.M.~de Oliveira Filho, Instituto de Matem\'atica e
  Estat\'istica, Universidade de S\~ao Paulo, Rua do Mat\~ao 1010,
  05508-090 S\~ao Paulo/SP, Brazil}
\email{fmario@gmail.com}

\author{Frank Vallentin} 
\address{F.~Vallentin, Mathematisches Institut, Universit\"at zu K\"oln, Weyertal 86--90, 50931 K\"oln, Germany}
\email{frank.vallentin@uni-koeln.de}

\date{November 15, 2015}

\maketitle

\section{Introduction}

During the last few years several new results on packing problems were
obtained using a blend of tools from semidefinite optimization,
polynomial optimization, and harmonic analysis. Schrijver
\cite{Schrijver2005a} used semidefinite optimization and the
Terwilliger algebra to obtain new upper bounds for binary codes,
Bachoc and Vallentin \cite{Bachoc2008a} used semidefinite optimization
and spherical harmonics for spherical codes, and Cohn and
Elkies~\cite{Cohn2003a} used linear optimization and Fourier analysis for
sphere packings leading to the breakthrough result of Cohn and Kumar
\cite{Cohn2009a}, who proved that the Leech lattice in dimension $24$
gives the best lattice sphere packing in its dimension. De Laat,
Oliveira and Vallentin \cite{Laat2014a} generalized the approach of
Cohn and Elkies to provide upper bounds for maximal densities of
packings of spheres having different radii. The most recent extension
is by Oliveira and Vallentin \cite{Oliveira2013a}, providing new upper
bounds for the density of packings of congruent copies of a given
convex body.

Typical in all this work is the use of semidefinite optimization and
harmonic analysis which gives newcomers to the field~---~often
overwhelmed with technical details~---~a hard time. Also typical is
that the computational challenge grows dramatically if one goes from
compact spaces, like binary Hamming space or the sphere, to
non-compact spaces like the Euclidean space.

Our goal in this paper is to provide an introduction to this topic in
an attempt to paint the big picture without loosing essential detail.
The paper is however not meant as a survey on results about geometric
packing problems --- this task would easily fill books! For a first
orientation we refer the interested reader to the now classical book
by Conway and Sloane \cite{Conway1988a}.

\section{Some history}
\label{sec:some history}

The \emph{sphere packing problem} asks: How much of three-dimensional
space can be filled with pairwise nonoverlapping translates of unit
spheres?  It was considered by Johannes Kepler (1571--1630) in his
work \textit{Strena seu de Nive Sexangula} (On the Six-Cornered
Snowflake) from 1611, which was his New Year's gift to his friend and
supporter Johann Matthäus Wacker von Wackenfels (1550--1619). He
explains the formation of snowflakes into crystals having sixfold
symmetry by drawing an analogy to dense sphere packings which possess
the same kind of symmetry. The general acceptance of atomism was yet
to come, so this explanation was a remarkable achievement. Kepler's
work is the first scientific writing about crystal formation; in it he
claims (essentially without any justification) that a specific
periodic structure, the face-centered cubic lattice, describes the
densest sphere packing having density $\pi/\sqrt{18} = 0.74\dots$. This
claim is now called Kepler's conjecture.

In 1998 Thomas Hales proved Kepler's conjecture; his proof makes heavy
use of computers and in 2009 he, together with his student Samuel
P.~Ferguson, was rewarded the Fulkerson Prize for his work.

The sphere packing problem, and more generally the problem of packing
copies of a given body, was also considered by David Hilbert. He
mentions it as part of his 18th problem:

\begin{quote}
18. Building up of Space from Congruent Polyhedra
\medskip

\noindent
(\dots) I point out the following question, related to the preceding
one, and important to number theory and perhaps sometimes useful to
physics and chemistry: How can one arrange most densely in space an
infinite number of equal solids of given form, e.g., spheres
with given radii or regular tetrahedra with given edges (or in
prescribed position), that is, how can one so fit them together that
the ratio of the filled to the unfilled space may be as great as
possible?
\end{quote}

The problem of packing congruent copies of regular tetrahedra,
mentioned by Hilbert, goes back to Aristotle's (384--322 BC)
refutation of a theory of Plato (428--348 BC), presented in the
\textit{Timaeus}, that claimed that each of the four elements had a
specific shape, namely one of the Platonic solids, and that the
properties of each element derived from its shape. So, for instance,
earth, the most stable and plastic element, is cubic in shape, and
fire, the acutest and most penetrating, has the shape of a
tetrahedron.

Aristotle presents several arguments against this theory in his
treatise \textit{De Caelo}. In one of his arguments (cf.~\textit{De
  Caelo}, Book~III, Chapter~VIII), he claims that it is irrational to
assign geometrical shapes to the four elements, since not all of space
can be thus filled. Indeed, says Aristotle, only the cube and the
pyramid (i.e., the regular tetrahedron) can fill space. So Aristotle's
argument uses the idea of the impossibility of a vacuum, together with
the fact that only two of the solids (corresponding to earth and fire)
can fill the whole of space, to refute Plato's theory.

Aristotle's claim that one can tile space with tetrahedra was picked
up by many of his commentators. Simplicius of Cilicia
(c.~490--c.~560), one of the main commentators of Aristotle in late
Antiquity, even states that, as eight cubes are sufficient to fill the
space around a given point, so are twelve regular tetrahedra (cf.\
page~42 in the translation by Mueller~\cite{Mueller}).

In the Middle Ages, Aristotle's Arabic commentator, Averroës
(1126--1198), restates the claim that twelve pyramids fill the space
around a point and gives an argument for it. Three planes meet at the
vertex of a cube, forming a so-called ``solid angle'' composed of
three right angles. Eight cubes fill the space around a point in
three-dimensional space, and these eight solid angles add up to a
total of~$8 \times 3$ right angles. Now, a solid angle of a
tetrahedron is composed of three angles of~$60^\circ$ each, totaling
two right angles. Since one needs~$8 \times 3$ right angles to fill
the space around a point (as can be seen from the cubes), and since~$8
\times 3 = 12 \times 2$, it follows that twelve tetrahedra fill the
space around a point.

Averroës' commentary introduced the problem to the medieval
schoolmen. Roger Bacon (c.~1214--1294) defended Averroës' position
against the claim that not twelve, but twenty tetrahedra are needed to
fill the space around a point. Thomas Bradwardine (c.~1290--1349)
disproved Averroës' claim with a very simple argument: If indeed it
would be possible to place twelve regular tetrahedra around a point in
such a way that no empty space results, then in addition to the five
Platonic solids, one would have another one, what is
impossible. According to him, those who argue that twenty tetrahedra
can be placed around a point have therefore a stronger position, since
one can obtain twenty pyramids by joining the bases of a regular
icosahedron to its center. Bradwardine observes that it is still
necessary to check whether the pyramids so obtained are regular or
not, but he leaves the question open (these pyramids are, as can be
shown using the construction of the icosahedron given in the
thirteenth book of Euclid's Elements, not regular).

The question was finally settled, it is believed, by Johannes Müller
von Königsberg (1436--1476), known as Regiomontanus, who proved that
it is impossible to tile space with regular tetrahedra.  Of
Regiomontanus' manuscript only the title, describing the contents of
the work, has been preserved, but there is no doubt he had all the
tools at his disposal to settle the problem.  Francesco Maurolico
(1494--1575) computed the angle between two faces of a regular
tetrahedron. This angle, equal to~$\arccos(1/3) \approx
70.52877^\circ$, is greater than~$60^\circ$ and smaller
than~$72^\circ$, and hence it follows that one cannot tile space with
tetrahedra. Maurolico's work has been recently rediscovered (see
Addabbo~\cite{Addabbo}). For more on the fascinating history of the
tetrahedra packing problem, including all the details presented here,
see the historical survey by Struik~\cite{Struik} and the survey by
Lagarias and Zong~\cite{Lagarias2012a}.

If one cannot tile space with regular tetrahedra, how much of space
can be filled with them?  Even today, the problem is far from being
solved. In~2006, Conway and Torquato~\cite{ConwayT} found surprisingly
dense packings of tetrahedra. This sparked renewed interest in the
problem and a race for the best construction (cf.~Lagarias and
Zong~\cite{Lagarias2012a} and Ziegler~\cite{Ziegler2011a}). The
current record is held by Chen, Engel, and Glotzer~\cite{ChenEG}, who
found in 2010 a packing with density~$\approx 0.8563$, a much larger
fraction of space than that which can be covered by spheres. This
prompted the quest for upper bounds; the current record rests with
Gravel, Elser, and Kallus~\cite{Gravel2011a}, who proved an upper
bound of~$1 - 2.6\ldots \cdot 10^{-25}$. They are themselves convinced
that the bound can be greatly improved: \medbreak

\begin{quote}
  In fact, we conjecture that the optimal packing density corresponds
  to a value of~$\delta$ [the fraction of empty space] many orders of
  magnitude larger than the one presented here. We propose as a
  challenge the task of finding an upper bound with a significantly
  larger value of~$\delta$ (e.g.,~$\delta > 0.01$) and the development
  of practical computational methods for establishing informative
  upper bounds.
\end{quote}

\section{Mathematical Modeling}
\label{sec:mathematical modeling}

How can one model mathematically the problem of packing spheres
or regular tetrahedra in~$\R^3$? Packing problems are
optimization problems and can be seen as infinite analogues of a
well-known problem in combinatorial optimization, namely the problem
of finding a maximum-weight independent set in a graph.
To see this, let us first consider two kinds of packing problems.

\begin{problem}[Translational body packings]
\label{prob:trans}
Given convex bodies~$\Kcal_1$, \dots,~$\Kcal_N \subseteq \R^n$, how
much of~$\R^n$ can be filled with pairwise nonoverlapping translated
copies of~$\Kcal_1$, \dots,~$\Kcal_N$?
\end{problem}

The sphere packing problem is then obtained by taking~$N = 1$ and
letting~$\Kcal_1$ be the unit ball.

\begin{problem}[Congruent body packings]
\label{prob:cong}
Given a convex body~$\Kcal \subseteq \R^n$, how much of~$\R^n$ can be
filled with pairwise nonoverlapping congruent (i.e., translated and
rotated) copies of~$\Kcal$?
\end{problem}

Here letting~$\Kcal$ be the unit ball gives the sphere packing
problem, and letting~$\Kcal$ be the regular tetrahedron gives the
tetrahedra packing problem. In a sense, Problem~\ref{prob:cong} is a
limiting case of Problem~\ref{prob:trans}: Given a convex body $\Kcal$
one tries to pack translative copies of infinitely many
rotations~$A\Kcal$ of~$\Kcal$, where $A \in \so(n)$ and~$\so(n)$
is the special orthogonal group of~$\R^n$ (i.e., the group of all
orthogonal $n \times n$ matrices with determinant~$1$).

We call a union of nonoverlapping (congruent or translated) copies of
some bodies a \defi{packing} of these bodies. In a packing bodies are
allowed to touch on their boundaries but not to intersect in their
interiors. The fraction of space covered by a packing is its
\defi{density}, so our goal is to find the maximum density of
packings. Here we are using an informal definition of density; later
in Section~\ref{sec:upper bounds} we will give a precise definition.

Let~$G = (V, E)$ be a graph, finite or infinite. An \defi{independent
  set} is a set~$I \subseteq V$ that does not contain adjacent
vertices. Packings of bodies can be seen as independent sets in some specially
defined graphs called packing graphs. Given convex bodies~$\Kcal_1$,
\dots,~$\Kcal_N \subseteq \R^n$, the \defi{translational packing
  graph} of~$\Kcal_1$, \dots,~$\Kcal_N$ is the graph~$G$ whose vertex
set is~$\{1, \ldots, N\} \times \R^n$. The vertices of~$G$ correspond
to possible choices of bodies in the packing: vertex~$(i, x)$
corresponds to placing the body~$x + \Kcal_i$ in the packing. This
interpretation defines the adjacency relation of~$G$: vertices~$(i,
x)$ and~$(j, y)$ are adjacent if the corresponding bodies overlap,
i.e., if
\[
(x + \Kcal_i)^\circ \cap (y + \Kcal_j)^\circ \neq \emptyset,
\]
where~$A^\circ$ is the interior of set~$A$.  So independent sets
of~$G$ correspond to packings of translated copies of~$\Kcal_1$,
\dots,~$\Kcal_N$ and vice versa.

A similar idea can be used regarding packings of congruent copies of a
given convex body~$\Kcal$. Given such a body, we consider its
\defi{congruent packing graph}, which is the graph~$G$ whose vertex
set is~$\so(n) \times \R^n$. The elements of~$\so(n)$
correspond to the possible rotations of~$\Kcal$, so that a vertex~$(A,
x)$ of~$G$ corresponds to placing the body~$x + A\Kcal$ in the
packing. Again, this gives the adjacency relation of~$G$:
vertices~$(A, x)$ and~$(B, y)$ are adjacent if
\[
(x + A\Kcal)^\circ \cap (y + B\Kcal)^\circ \neq \emptyset.
\]
With this, independent sets of~$G$ correspond to packings of congruent
copies of~$\Kcal$ and vice versa.

Packings therefore correspond to independent sets of the packing
graph. If we measure the weight of an independent set by the density of
the associated packing, then Problems~\ref{prob:trans}
and~\ref{prob:cong} ask us to find maximum-weight independent sets in
the corresponding packing graphs.

Does this modeling help? Finding a maximum cardinality independent
set in a finite graph is a well-known NP-hard problem, figuring in
Karp's list of~21 problems. Many techniques have been developed in
combinatorial optimization to deal with hard problems: the basic
approach is that one tries to develop efficient methods to find lower
and upper bounds. In the case of the maximum-cardinality independent
set problem, lower bounds are constructive and come from heuristics
that try to find independent sets of large size. Analogously, for
packing problems one has the adaptive shrinking cell scheme of
Torquato and Jiao~\cite{Torquato2009a}, which can successfully generate
dense packings. 

As for upper bounds, Lovász~\cite{Lovasz1979a} introduced a graph
parameter, the theta number, that provides an upper bound for the
maximum cardinality of independent sets of a finite graph; Lovász's
theta number can moreover be computed efficiently using semidefinite
optimization. The most successful approaches to obtain upper bounds
for the maximum densities of packings all use extensions of the theta
number. The theta number can be quite naturally extended to graphs
having compact vertex sets, as we show in
Section~\ref{sec:lovasz-ext}; still, this extension cannot be applied
to the packing graphs we described above, because they have noncompact
vertex sets. These graphs can be compactified, however, as we discuss
in Section~\ref{sec:upper bounds}, and then the extension of the theta
number can be applied.

\section{The Lovász theta number and an extension}
\label{sec:lovasz-ext}

The \defi{independence number} of a graph~$G = (V, E)$ (finite or
infinite) is the graph parameter
\[
\alpha(G) = \max\{\, |I| : \text{$I$ is independent}\,\}.
\]
Given a nonnegative weight function~$w\colon V \to \R_+$, one may also
define the \defi{weighted independence number} of~$G$ as
\[
\alpha_w(G) = \max\{\, w(I) : \text{$I$ is independent}\,\},
\]
where~$w(I) = \sum_{x \in I} w(x)$. Weights will be useful in packing
problems because, when we want to pack different kinds of bodies, like
spheres having different radii, the weight function allows us to
distinguish between big and small bodies.

The theta number introduced by Lovász~\cite{Lovasz1979a} provides an
upper bound to the independence number of a graph. It was later
strengthened and extended to the weighted case by Grötschel, Lovász,
and Schrijver~\cite{Groetschel1981a}.  There are many equivalent ways
of defining their graph parameter; the one most convenient for us is
the following. Given a finite graph~$G = (V, E)$ and a weight
function~$w\colon V \to \R_+$, we define
\begin{equation}
\label{eq:theta}
\begin{array}{rll}
\vartheta'_w(G) = \min&M\\
&K(x, x) \leq M&\text{for all~$x \in V$},\\
&K(x, y) \leq 0&\text{for all~$\{x, y\} \notin E$ with~$x \neq y$},\\
&\multicolumn{2}{l}{\text{$K \in \R^{V \times V}$ is symmetric},}\\
&\multicolumn{2}{l}{\text{$K - (w^{1/2})(w^{1/2})^\tp$ is positive semidefinite},}
\end{array}
\end{equation}
where~$w^{1/2} \in \R^V$ is such that~$w^{1/2}(x) = w(x)^{1/2}$.

\begin{theorem}
\label{thm:theta}
Let~$G = (V, E)$ be a finite graph and~$w\colon V \to \R_+$ be a
weight function. Then~$\alpha_w(G) \leq \vartheta'_w(G)$.
\end{theorem}

\begin{proof}
Let~$I \subseteq V$ be an independent set such that~$w(I) > 0$ (if
there is no such independent set, then~$\alpha_w(G) = 0$, and the
theorem follows trivially) and let~$M$ and~$K$ be a feasible solution
of~\eqref{eq:theta}.

Consider the sum
\[
\sum_{x, y \in I} w(x)^{1/2} w(y)^{1/2} K(x, y).
\]
This sum is at least
\[
\sum_{x, y \in I} w(x)^{1/2} w(y)^{1/2} w(x)^{1/2} w(y)^{1/2}
= w(I)^2
\]
because~$K - (w^{1/2})(w^{1/2})^\tp$ is positive semidefinite.

The same sum is also at most
\[
\sum_{x \in I} w(x) K(x, x) \leq M w(I)
\]
because~$K(x, x) \leq M$ and~$K(x, y)  \leq 0$ for distinct~$x$, $y
\in I$. Combining both inequalities proves the theorem.
\end{proof}

Notice that the formulation we use for~$\vartheta'_w(G)$ is a
\emph{dual formulation}, since any feasible solution gives an upper
bound for the independence number.

So~$\vartheta'_w(G)$ provides an upper bound for~$\alpha_w(G)$
when~$G$ is finite. When more generally~$V$ is a separable and compact
measure space satisfying a mild technical condition, graph
parameter~$\vartheta'_w$ can be extended in a natural way so as to
provide an upper bound for the weighted independence number.

This extension relies on a basic notion from functional analysis, that
of kernel. Let~$V$ be a separable and compact topological space
and~$\mu$ be a finite Borel measure over~$V$. A kernel is a 
complex-valued function~$K \in L^2(V \times V)$.

A kernel~$K$ can be seen as a generalization of a matrix. Like a matrix,
a kernel defines an operator on~$L^2(V)$ by
\[
(K f)(x) = \int_V K(x, y) f(y)\, d\mu(y).
\]

Kernel~$K$ is \defi{Hermitian} if~$K(x, y) = \overline{K(y, x)}$ for all~$x$,
$y \in V$. Hermitian kernels are the analogues of Hermitian matrices,
and an analogue of the spectral decomposition theorem, known as the
Hilbert-Schmidt theorem, holds, as we describe now.

A function~$f \in L^2(V)$, $f \neq 0$, is an \defi{eigenfunction}
of~$K$ if $K f = \lambda f$ for some number~$\lambda$, which is the
\defi{associated eigenvalue} of~$f$. We say~$\lambda$ is an eigenvalue
of~$K$ if it is the associated eigenvalue of some eigenfunction
of~$K$.  The Hilbert-Schmidt theorem states that, for a Hermitian
kernel~$K$, there is a complete orthonormal
system~$\varphi_1$, $\varphi_2$, \dots\ of~$L^2(V)$ consisting of
eigenfunctions of~$K$ such that
\[
K(x, y) = \sum_{i=1}^\infty \lambda_i \varphi_i(x) \overline{\varphi_i(y)}
\]
with~$L^2$ convergence, where the real number~$\lambda_i$ is the
associated eigenvalue of~$\varphi_i$. Then the~$\lambda_i$ with their
multiplicities are all the eigenvalues of~$K$.

A Hermitian kernel~$K$ is \defi{positive} if all its
eigenvalues are nonnegative; this is the analogue of a positive
semidefinite matrix. An equivalent definition is as follows: $K$ is
positive if for every~$\rho \in L^2(V)$ we have
\[
\int_V \int_V K(x, y) \rho(x) \overline{\rho(y)}\, d\mu(x) d\mu(y) \geq 0.
\]

Using kernels, one may extend the definition of~$\vartheta'_w$ also to
graphs defined over separable and compact measure spaces, simply by
replacing the matrices in~\eqref{eq:theta} by continuous kernels. In
other words we define
\begin{equation}
\label{eq:theta-gen}
\begin{array}{rll}
\vartheta'_w(G) = \inf&M\\
&K(x, x) \leq M&\text{for all~$x \in V$},\\
&K(x, y) \leq 0&\text{for all~$\{x, y\} \notin E$ with~$x \neq y$},\\
&\multicolumn{2}{l}{\text{$K\colon V \times V \to \R$ is continuous and
    Hermitian},}\\
&\multicolumn{2}{l}{\text{$K - W$ is positive},}\\
\end{array}
\end{equation}
where~$W \in L^2(V \times V)$ is the kernel such that~$W(x, y) =
w(x)^{1/2} w(y)^{1/2}$.

One then has the theorem:

\begin{theorem}
\label{thm:theta-gen}
  Let~$G = (V, E)$ be a graph where~$V$ is a separable and compact
  measure space in which any open set has nonzero
  measure. Let~$w\colon V \to \R_+$ be a continuous weight
  function. Then~$\alpha_w(G) \leq \vartheta'_w(G)$.
\end{theorem}

\begin{proof}
  Since~$V$ is compact and since we assume that every open subset
  of~$V$ has nonzero measure, we may use the following observation of
  Bochner~\cite{Bochner}: a continuous kernel~$K$ is positive if and
  only if for any choice of~$N$ and points~$x_1$, \dots,~$x_N \in V$
  we have that the matrix
\[
\bigl(K(x_i, x_j)\bigr)_{i,j=1}^N
\]
is positive semidefinite.

Using this characterization of continuous and positive kernels, we may
mimic the proof of Theorem~\ref{thm:theta} and obtain the desired
result. This is why we require~$K$ to be continuous in the definition
of~$\vartheta'_w$ and also why we require~$w$ to be a continuous
function: because we want to apply Bochner's characterization to~$K -
W$.
\end{proof}

As was the case with Theorem~\ref{thm:theta}, any feasible solution
of~\eqref{eq:theta-gen} gives an upper bound for the weighted
independence number. This is actually quite useful in the infinite
setting, because then it is often harder to obtain optimal solutions.
Notice that it might also be that~$\alpha_w(G) = \infty$. In this
case, the theorem still holds, since~\eqref{eq:theta-gen} will be
infeasible, and therefore~$\vartheta'_w(G) = \infty$.

\section{Exploiting symmetry with harmonic analysis}
\label{sec:harmonic analysis}

If $G$ is a finite graph, then computing $\vartheta'_w(G)$ is solving
a semidefinite program whose value can be found with the help of a
computer, i.e., it can be approximated up to arbitrary precision in
polynomial time. This is a theoretical assertion however; in practice,
for moderately big graphs (say with thousands of vertices), if one
cannot exploit any special structure of the graph, then computing the
theta number is often impossible with today's methods and computers.

If the graph $G$ is infinite, we are dealing with an
infinite-dimensional semidefinite program. If one then desires to use
computational optimization methods, at some point the transition from
infinite to finite has to be made.  One way to make this transition is
to use finer and finer grids to discretize the infinite graph and
solve the corresponding finite semidefinite programs, obtaining bounds
for the infinite problem. For coarse grids, however, this approach
performs poorly, and for fine grids it becomes soon computationally
infeasible. Moreover, with this approach one looses the entire
geometrical structure of the packing graphs.

The alternative is to use harmonic analysis. Instead of
computing~$\vartheta'_w$ in the ``time domain'', we could formulate the
optimization problem in the ``Fourier domain''. This has a twofold
advantage. First, the Fourier domain can be discretized essentially by
truncation and in doing so we do not loose too much, since it is to be
expected that most of the information in a well-structured problem
(like a packing problem) is to be concentrated in the beginning of the
spectrum. Second, the translation group $\R^n$ acts on the
translational packing graph and the group of Euclidean motions $\so(n)
\rtimes \R^n$ acts on the congruent packing graph; using harmonic
analysis we can exploit the symmetry of this situation.  On the down
side, a very explicit understanding of the harmonic analysis of these
two groups is needed, what in the case of the motion group can be
cumbersome.

To make things concrete, let us demonstrate the basic strategy using
the cyclic group~$\Zn$. This group is finite, so that discretization
is unnecessary, and Abelian, so that harmonic analysis becomes
simple. Nevertheless, this simple example already carries many
essential features, and ought to be kept in mind by the reader when
the more complicated cases are treated later.

Let $\Sigma \subseteq \Zn$ with $0 \not\in \Sigma$ be closed under
taking negatives, i.e., $\Sigma = -\Sigma$. Then we define the
\defi{Cayley graph}
\[
\cayley(\Zn, \Sigma) = (\Zn, \{\,\{x,y\} : x - y \in \Sigma\}\,\}),
\]
which is an undirected graph whose vertices are the elements of~$\Zn$
and where $\Sigma$ defines the neighborhood of the neutral
element~$0$; this neighborhood is then transported to every vertex by
group translations. Since $\Sigma = -\Sigma$, the definition is
consistent, and since~$0 \notin \Sigma$, the Cayley graph does not
have loops.  For example, the $n$-cycle can be represented as a Cayley
graph:
\[
C_n = \cayley(\Zn, \Sigma) \quad \text{with} \quad \Sigma = \{1,-1\}.
\]

The goal in this section is to show that the computation of the theta
number $\vartheta'_e(\cayley(\Zn, \Sigma))$ with unit weights $e =
(1, \ldots, 1)$ reduces from a semidefinite program to a linear
program if one works in the Fourier domain.

For this we need the \defi{characters} of $\Zn$, which are group
homomorphisms $\chi\colon \Zn \to \T$, where~$\T$ is the unit circle in
the complex plane. So every character~$\chi$ satisfies
\[
\chi(x + y) = \chi(x) \chi(y)
\]
for all~$x$, $y \in \Zn$.

The characters themselves form a group with the operation of pointwise
multiplication $(\chi \psi)(x) = \chi(x) \psi(x)$; this is the
\defi{dual group}~$\dZn$ of~$\Zn$. The \defi{trivial character}~$e$
of~$\Zn$ defined by~$e(x) = 1$ for all~$x \in \Zn$ is the unit
element. Moreover, if~$\chi$ is a character, then its inverse is its
complex conjugate~$\overline{\chi}$ that is such
that~$\overline{\chi}(x) = \overline{\chi(x)}$ for all~$x \in \Zn$.
We often view characters as vectors in the vector space~$\C^{\Zn}$.

\begin{lemma}
  Let $\chi$ and $\psi$ be characters of~$\Zn$. Then the following
  orthogonality relation holds:
\[
\chi^* \psi = \sum_{x \in \Zn} \overline{\chi(x)} \psi(x) = 
\begin{cases}
|\Zn| & \text{if $\chi = \psi$,}\\
0 & \text{otherwise.}
\end{cases}
\]
\end{lemma}

\begin{proof}
If $\chi = \psi$, then,
\[
\chi^* \chi = \sum_{x \in \Zn} \overline{\chi}(x) \chi(x) = \sum_{x \in \Zn} 1 = |\Zn|
\]
holds. If $\chi \neq \psi$, then there is $y \in \Zn$, so that
$(\overline{\chi}\psi)(y) \neq 1$. Furthermore, we have
\begin{multline*}
(\overline{\chi}\psi)(y) \chi^* \psi  = (\overline{\chi}\psi)(y) \sum_{x \in \Zn} \overline{\chi}(x) \psi(x)
= \sum_{x \in \Zn} \overline{\chi}(x + y) \psi(x + y)\\
= \sum_{x \in \Zn} \overline{\chi}(x) \psi(x)
= \chi^* \psi,
\end{multline*}
so $\chi^* \psi$ has to be zero.
\end{proof}

As a corollary we can explicitly give all characters of~$\Zn$ and see
that they form an orthogonal basis of $\C^{\Zn}$. It follows that the
dual group $\dZn$ is isomorphic to~$\Zn$.

\begin{corollary}
Every element~$u \in \Zn$ defines a character of~$\Zn$ by
\[
\chi_u(x) = e^{2\pi i u x/n}.
\]
The map $u \mapsto \chi_u$ is a group isomorphism between $\Zn$ and its
dual group $\dZn$.
\end{corollary}

\begin{proof}
  One immediately verifies that the map $u \mapsto \chi_u$ is
  well-defined, that it is an injective group homomorphism, and that
  $\chi_u$ is a character of~$\Zn$. By the orthogonality relation we
  see that the number of different characters of~$\Zn$ is at most the
  dimension of the space $\C^{\Zn}$, hence~$|\Zn|$ equals~$|\dZn|$ and
  the map is a bijection.
\end{proof}

Given a function~$f\colon \Zn \to \C$, the function~$\hat{f}\colon \dZn \to \C$ such that
\[
\hat{f}(\chi) = \frac{1}{|\Zn|}\sum_{x \in \Zn} f(x) \chi^{-1}(x)
\]
is the \defi{discrete Fourier transform} of~$f$; the coefficients
$\hat{f}(\chi)$ are called the \defi{Fourier coefficients} of~$f$. We
have then the \defi{Fourier inversion formula}:
\[
f(x) = \sum_{\chi \in \dZn} \hat{f}(\chi) \chi(x).
\]

We say that~$f\colon \Zn \to \C$ is of \defi{positive type} if~$f(x) =
\overline{f(-x)}$ for all~$x \in \Zn$ and for all~$\rho\colon \Zn \to
\C$ we have
\[
\sum_{x,y \in \Zn} f(x-y) \rho(x) \overline{\rho(y)} \geq 0.
\]
So~$f$ is of positive type if and only if the matrix~$K(x, y)
= f(x - y)$ is positive semidefinite. With this we have the following
characterization for the theta number of~$\cayley(\Zn, \Sigma)$.

\begin{theorem}
\label{thm:symmetryreduction}
We have that
\begin{equation}
\label{eq:time-domain}
\begin{array}{rll}
\vartheta'_e(\cayley(\Zn, \Sigma)) = \min&f(0)\\
&f(x) \leq 0\quad\text{for all~$x \notin \Sigma \cup \{0\}$,}\\
&\sum_{x\in\Zn} f(x) \geq |\Z_n|,\\
&f\colon \Zn \to \R\text{ is of positive type.}
\end{array}
\end{equation}

Alternatively, expressing~$f$ in the Fourier domain we obtain:
\begin{equation}
\label{eq:fourier-domain}
\begin{array}{rll}
\vartheta'_e(\cayley(\Zn, \Sigma)) = \min&\sum_{\chi \in \dZn} \hat{f}(\chi)\\
&\sum_{\chi \in \dZn} \hat{f}(\chi) \chi(x) \leq 0\quad\text{for all~$x \notin \Sigma \cup \{0\}$,}\\
&\hat{f}(e) \geq 1,\\
&\hat{f}(\chi) \geq 0\text{ and }\hat{f}(\chi)
  = \hat{f}(\chi^{-1})\text{ for all~$\chi \in \dZn$.}
\end{array}
\end{equation}
\end{theorem}

\begin{proof}
Functions~$f\colon \Zn \to \C$ correspond to $\Zn$-invariant
matrices~$K\colon \Zn \times \Zn \to \C$, which are matrices such
that~$K(x+z, y+z) = K(x, y)$ for all~$x$, $y$, $z \in \Zn$.

In solving problem~\eqref{eq:theta} for computing~$\vartheta'_e$
we may restrict ourselves to $\Zn$-invariant matrices. This can be
seen via a symmetrization argument: If~$(M, K)$ is an optimal solution
of~\eqref{eq:theta}, then so is $(M, \overline{K})$ with
\[
\overline{K}(x, y) = \frac{1}{|\Zn|} \sum_{z\in\Zn} K(x+z, y+z),
\]
which is $\Zn$-invariant.

So we can translate problem~\eqref{eq:theta}
into~\eqref{eq:time-domain}. The objective function and the constraint
on nonedges translate easily. The positive-semidefiniteness constraint
requires a bit more work.

First, observe that to require~$K$ to be real and symmetric is to
require~$f$ to be real and such that~$f(x) = f(-x)$ for all~$x \in
\Zn$. We claim that each character~$\chi$ of~$\Zn$ gives an
eigenvector of~$K$ with eigenvalue~$|\Zn| \hat{f}(\chi)$. Indeed, using
the inversion formula we have
\[
\begin{split}
(K\chi)(x) &= \sum_{y \in \Zn} K(x, y) \chi(y) = \sum_{y\in\Zn} f(x-y) \chi(y)\\
&=\sum_{y\in\Zn} \sum_{\psi\in\dZn} \hat{f}(\psi) \psi(x-y) \chi(y)\\
&=\sum_{\psi\in\dZn} \hat{f}(\psi) \sum_{y\in\Zn} \psi(y) \chi(x-y)\\
&=\sum_{\psi\in\dZn} \hat{f}(\psi) \chi(x) \sum_{y\in\Zn} \psi(y)
  \overline{\chi(y)}\\
&=|\Zn| \hat{f}(\chi) \chi(x),
\end{split}
\]
as claimed.

This immediately implies that~$K$ is positive semidefinite --- or,
equivalently,~$f$ is of positive type --- if and only
if~$\hat{f}(\chi) \geq 0$ for all characters~$\chi$. Now,
since~$\hat{f}(e) = |\Zn|^{-1}\sum_{x\in\Zn} f(x)$, and since~$e$ is
an eigenvalue of~$K$, then~$K - ee^\tp$ is positive semidefinite if
and only if~$\sum_{x\in\Zn} f(x) \geq |\Zn|$ and~$f$ is of positive
type.

So we see that~\eqref{eq:theta} can be translated
into~\eqref{eq:time-domain}. Using the inversion formula and noting
that~$f$ is real-valued if and only if~$\hat{f}(\chi) =
\hat{f}(\chi^{-1})$ for all~$\chi$, one immediately
obtains~\eqref{eq:fourier-domain}.
\end{proof}

Cayley graphs on the cyclic group are not particularly
exciting. Everything in this section, however, can be
straightforwardly applied to any finite Abelian group. If, for
instance, one considers the group~$\Z_2^n$, then it becomes possible
to model binary codes as independent sets of Cayley graphs, and the
analogue of Theorem~\ref{thm:symmetryreduction} gives Delsarte's
linear programming bound~\cite{Delsarte}.

\section{Upper bounds for congruent and translational body packings}
\label{sec:upper bounds}

The packing graphs described above have noncompact vertex sets, but we
said they can be compactified so that the theta number can be
applied. Let us now see how that can be done.

First we need a definition of packing density. Given a
packing~$\Pcal$, we say that its density is~$\Delta$ if for every~$p
\in \R^n$ we have
\[
\Delta = \lim_{r \to \infty} \frac{\vol(B(p, r) \cap \Pcal)}{\vol B(p, r)},
\]
where~$B(p, r)$ is the ball of radius~$r$ centered at~$p$. Not every
packing has a density, but every packing has an \defi{upper density}
given by
\[
\limsup_{r \to \infty} \sup_{p \in \R^n}
\frac{\vol(B(p, r) \cap \Pcal)}{\vol B(p, r)}.
\]

We say that a packing~$\Pcal$ is \defi{periodic} if there is a
lattice\footnote{A lattice is a discrete subgroup of~$(\R^n, +)$.} $L
\subseteq \R^n$ that leaves~$\Pcal$ invariant, that is, $\Pcal = x +
\Pcal$ for every~$x \in L$. Lattice~$L$ is the \defi{periodicity
  lattice} of~$\Pcal$. In other words,~$\Pcal$ consists of some bodies
placed inside the fundamental cell of~$L$, and this arrangement
repeats itself at each copy of the fundamental cell translated by
vectors of the lattice.

Periodic packings always have a density. Moreover, given any
packing~$\Pcal$, one may define a sequence of periodic packings whose
fundamental cells have volumes approaching infinity and whose
densities converge to the upper density of~$\Pcal$. So, when computing
bounds for the maximum density of packings, we may
restrict ourselves to periodic packings.

This is the key observation that allows us to compactify the packing
graphs. For let~$\Kcal_1$, \dots,~$\Kcal_N \subseteq \R^n$ be some
given convex bodies. We have defined the translational packing graph
of~$\Kcal_1$, \dots,~$\Kcal_N$. Given a lattice~$L \subseteq \R^n$, we
may define a periodic version of the packing graph. This is the
graph~$G_L$, whose vertex set is~$V = \{ 1, \dots, N \} \times (\R^n /
L)$. Now, vertex~$(i, x)$ of~$G_L$ corresponds not only to one body,
but to many: it corresponds to all the bodies~$x + v + \Kcal_i$,
for~$v \in L$. Vertices~$(i, x)$ and~$(j, y)$ are then adjacent if for
some~$v \in L$ we have
\[
(x + v + \Kcal_i)^\circ \cap (y + \Kcal_j)^\circ \neq \emptyset.
\]
Then an independent set of~$G_L$ corresponds to a periodic packing of
translations of~$\Kcal_1$, \dots,~$\Kcal_N$ with periodicity
lattice~$L$, and vice versa.

Graph $G_L$ has a compact vertex set and each one of its independent
sets is finite. If we consider the weight function~$w\colon V \to
\R_+$ such that~$w(i, x) = \vol \Kcal_i$, then the maximum density of
a periodic packing with periodicity lattice~$L$ is given by
\[
\frac{\alpha_w(G_L)}{\vol(\R^n / L)}.
\]
So one strategy to find an upper bound for the maximum density of a
packing is to find an upper bound for~$\alpha_w(G_L)$ for every~$L$.

Notice~$V$ is actually a separable and compact measure space that
satisfies the hypothesis of
Theorem~\ref{thm:theta-gen}. So~$\vartheta'_w(G_L)$ provides an upper
bound for~$\alpha_w(G_L)$.  Let us see how one may obtain a feasible
solution of~\eqref{eq:theta-gen} for every graph~$G_L$.

Let~$f\colon \R^n \to \C$ be a rapidly decreasing function. This is an
infinitely differentiable function with the following property: any
derivative, multiplied by any polynomial, is a
bounded function.

The \defi{Fourier transform} of~$f$ computed at~$u \in \R^n$ is
\[
\hat{f}(u) = \int_{\R^n} f(x) e^{-2\pi i u \cdot x}\, dx,
\]
where~$u \cdot x = u_1 x_1 + \cdots + u_n x_n$. Since~$f$ is rapidly
decreasing, the inversion formula holds, giving
\[
f(x) = \int_{\R^n} \hat{f}(u) e^{2\pi i u\cdot x}\, du.
\]

Consider now a matrix-valued function~$f\colon \R^n \to \C^{N \times
  N}$, where $f(x) = \bigl(f_{ij}(x)\bigr)_{i,j=1}^N$ and each
function~$f_{ij}$ is rapidly decreasing. For~$u \in \R^n$ we write
\[
\hat{f}(u) = \bigl(\hat{f}_{ij}(u)\bigr)_{i,j=1}^N.
\]
So the Fourier transform of~$f$ is also a matrix-valued function.

We say~$f$ is of \defi{positive type} if~$f(x) = f(-x)^*$ for every~$x
\in \R^n$ and for every $L^\infty$ function $\rho \colon \R^n \to
\C^N$ we have
\[
\int_{\R^n} \int_{\R^n} \rho(y)^* f(x - y) \rho(x)\, dx dy \geq 0.
\]
One may prove that~$f$ is of positive type if and only if~$\hat{f}(u)$
is positive semidefinite for every~$u \in \R^n$.

With this we have the following theorem:

\begin{theorem}
\label{thm:trans-pack}
Let~$\Kcal_1$, \dots,~$\Kcal_N \subseteq \R^n$ be convex
bodies. Suppose~$f\colon \R^n \to \R^{N \times N}$ is such
that each~$f_{ij}$ is rapidly decreasing and that it satisfies the
following conditions:

\begin{enumerate}
\item[(i)] $f_{ij}(x) \leq 0$ whenever $(x + \Kcal_i)^\circ \cap \Kcal_j^\circ =
  \emptyset$;

\item[(ii)] $\hat{f}(0) - \bigl((\vol \Kcal_i)^{1/2} (\vol
  \Kcal_j)^{1/2}\bigr)_{i,j=1}^N$ is positive semidefinite;

\item[(iii)] $f$ is of positive type.
\end{enumerate}

\noindent
Then the maximum density of a packing of translated copies
of~$\Kcal_1$, \dots,~$\Kcal_N$ is at most~$\max\{\, f_{ii}(0) :
\text{$i = 1$, \dots,~$N$}\,\}$.
\end{theorem}

\begin{proof}
Let~$w\colon V \to \R_+$ be the weight function such that~$w(i, x) =
\vol \Kcal_i$ for all~$(i, x) \in V$. The proof of the theorem
consists in deriving from~$f$, for \textit{every} lattice~$L \subseteq \R^n$, a
kernel~$K_L \in L^2(V \times V)$, where~$V = \{ 1, \ldots, N \} \times
(\R^n / L)$, and a number~$M_L$ that together give a feasible solution
of~\eqref{eq:theta-gen}, thus obtaining an upper bound
for~$\alpha_w(G_L)$.

For a given lattice~$L$, we let
\[
K_L((i, x), (j, y)) = \vol(\R^n / L) \sum_{v \in L} f_{ij}(x - y + v).
\]
The above sum is well-defined since each~$f_{ij}$ is rapidly
decreasing. Moreover, this also implies that~$K_L$ is continuous.

Given two distinct, nonadjacent vertices~$(i, x)$ and~$(j, y)$
of~$G_L$, we have that for all~$v \in L$,
\[
(x + v + \Kcal_i)^\circ \cap (y + \Kcal_j)^\circ = \emptyset
\iff
(x - y + v + \Kcal_i)^\circ \cap \Kcal_j^\circ = \emptyset.
\]
This means that~$f_{ij}(x - y + v) \leq 0$ for all~$v \in L$. But
then $K_L((i, x), (j, y)) \leq 0$, as we wanted.

Now we show that~$K_L - W$ is a positive kernel where
$W((i,x), (j,y)) = (\vol \Kcal_i)^{1/2} (\vol \Kcal_j)^{1/2}$. This is
implied by conditions~(ii) and~(iii) of the theorem, and can be proven
directly with a bit of work by combining the definition of a positive
kernel with that of a function of positive type. We take however
another road and exhibit a complete list of eigenfunctions and
eigenvalues of~$K_L - W$.

Let~$L^* = \{\, u \in \R^n : \text{$u \cdot v \in \Z$ for all~$v \in
  L$}\,\}$ be the \defi{dual lattice} of~$L$ and consider the
matrix~$W' \in \R^{N \times N}$ with~$W'_{ij} = (\vol \Kcal_i)^{1/2}
(\vol \Kcal_j)^{1/2}$.  Since~$f$ is of positive type, for each~$u \in
\R^n$ we have that~$\hat{f}(u)$ is positive semidefinite. Moreover,
from condition~(ii) we have that~$\hat{f}(0) - W'$ is positive
semidefinite. So the matrices
\[
\hat{f}(u) - \delta_u W',
\]
where~$\delta_u$ equals~$1$ if~$u = 0$ and~$0$ otherwise, are positive
semidefinite.

For~$u \in L^*$, let~$a_{1,u}$, \dots,~$a_{N,u}$ be an orthonormal basis
of~$\R^N$ consisting of eigenvectors of~$\hat{f}(u) - \delta_u W'$,
with associated eigenvalues~$\lambda_{1, u}$, \dots,~$\lambda_{N, u}$,
which are all nonnegative.

Also for~$u \in L^*$, let~$\chi_u(x) = e^{2\pi i u \cdot
  x}$. Then~$(\vol(\R^n / L))^{1/2} \chi_u$, $u \in L^*$, forms a
complete orthonormal system of~$L^2(\R^n / L)$, and so
\[
(\vol(\R^n / L))^{1/2} a_{k, u} \otimes \chi_u
\]
for~$k = 1$, \dots,~$N$ and~$u \in L^*$ forms a complete orthonormal
system of~$L^2(V)$. We claim that each such function is an
eigenfunction of~$K_L - W$.

Indeed, let~$(i, x) \in V$ be given. Notice that
\[
\begin{split}
[W(a_{k,u} \otimes \chi_u)]((i, x))
&=\int_V W((i, x), (j, y)) (a_{k,u} \otimes \chi_u)(j, y)\,
d(j, y)\\
&=\sum_{j=1}^N W'_{ij} (a_{k,u})_j \int_{\R^n/L} e^{2\pi i u \cdot
  y}\, dy\\
&=\sum_{j=1}^N W'_{ij} (a_{k,u})_j \vol(\R^n / L) \delta_u\\
&=\vol(\R^n / L) (W' a_{k,u})_i \delta_u.
\end{split}
\]

Similarly we have
\[
\begin{split}
[K_L(a_{k,u} \otimes \chi_u)](i, x)
&=\int_V K_L((i, x), (j, y)) (a_{k,u} \otimes \chi_u)(j, y)\, d(j, y)\\
&=\vol(\R^n / L) \sum_{j=1}^N \int_{\R^n / L} \sum_{v \in L} f_{ij}(x-y+v) (a_{k,u})_j
e^{2\pi i u \cdot y}\, dy\\
&=\vol(\R^n / L)\sum_{j=1}^N (a_{k,u})_j \int_{\R^n} f_{ij}(x - y) e^{2\pi i u \cdot
  y}\, dy\\
&=\vol(\R^n / L)\sum_{j=1}^N (a_{k,u})_j \int_{\R^n} f_{ij}(y) e^{2\pi i u \cdot
  (x-y)}\, dy\\
&=\vol(\R^n / L)\sum_{j=1}^N \hat{f}_{ij}(u) (a_{k,u})_j e^{2\pi i u
  \cdot x}\\
&=\vol(\R^n / L) (\hat{f}(u) a_{k,u})_i e^{2\pi i u \cdot x}.
\end{split}
\]

Putting it all together we have
\[
\begin{split}
[(K_L - W)(a_{k,u} \otimes \chi_u)](i, x)
&= \vol(\R^n / L) (\hat{f}(u) a_{k,u} - \delta_u W' a_{k, u})_i e^{2\pi
  i u \cdot x}\\
&= \vol(\R^n / L) \lambda_{k,u} (a_{k,u})_i e^{2\pi i u \cdot x}\\
&= \vol(\R^n / L) \lambda_{k,u} (a_{k,u} \otimes \chi_u)(i, x).
\end{split}
\]

So we see that all the functions~$a_{k,u} \otimes \chi_u$ are
eigenfunctions of~$K_L - W$ with nonnegative associated eigenvalues,
and it follows that~$K_L - W$ is a positive kernel.

Finally, we need to provide the bound~$M_L$ on the diagonal elements
of~$K_L$. To do that, we assume that the minimum vector of~$L$ is
large enough so that~$(v + \Kcal_i)^\circ \cap \Kcal_i^\circ =
\emptyset$ for all nonzero~$v \in L$; this is no loss of generality,
since we care only about lattices with large fundamental cell, and one
can scale~$L$ appropriately. But this means that~$f_{ii}(v) \leq 0$
for all~$i$ and nonzero~$v \in L$. Then from the definition of~$K_L$
we have that
\[
K_L((i, x), (i, x)) \leq \vol(\R^n/L) f_{ii}(0)
\]
for all~$(i, x) \in V$, and we can take~$M_L = \vol(\R^n/L)\max\{\, f_{ii}(0) :
\text{$i = 1$, \dots,~$N$}\,\}$.

So we have that the maximum density of a periodic packing with
periodicity lattice~$L$ is
\[
\frac{\alpha_w(G_L)}{\vol(\R^n / L)} \leq
\frac{\vartheta'_w(G_L)}{\vol(\R^n / L)} \leq \max\{\, f_{ii}(0) :
\text{$i = 0$, \dots,~$N$}\,\},
\]
proving the theorem.
\end{proof}

Theorem~\ref{thm:trans-pack} was stated in the time domain, but using
the inversion formula and the fact that a matrix-valued function is of
positive type if and only if its Fourier transform is everywhere
positive-semidefinite, we can restate it in the Fourier domain. We
will use this alternative version in the next section, when we discuss
a computational approach to find functions~$f$ satisfying the
conditions required in the theorem.

When~$N = 1$, Theorem~\ref{thm:trans-pack} is a direct analogue of
Theorem~\ref{thm:symmetryreduction}. Indeed, then the translational
packing graph is actually a Cayley graph with~$\R^n$ as its vertex
set. Though noncompact,~$\R^n$ is an Abelian group, and the
functions~$\chi_u(x) = e^{2\pi i u \cdot x}$, for~$u \in \R^n$, give its
characters. The Fourier transform for~$\R^n$ is the
direct analogue of the discrete Fourier transform for~$\Zn$.

Moreover, except for the compactification step and other technical
issues stemming from analysis, the proof of
Theorem~\ref{thm:trans-pack} follows exactly the same pattern of the
proof of Theorem~\ref{thm:symmetryreduction}. Notice in particular how
the characters give eigenvectors of the translation-invariant
kernel~$K$ defined by~$f$.

A theorem similar to Theorem~\ref{thm:trans-pack} can be proven for
packings of congruent copies of a given convex body~$\Kcal \subseteq
\R^n$. Recall that the congruent packing graph has as vertex set~$V =
\so(n) \times \R^n$. Set~$V$ is actually a group with identity~$(I,
0)$, where~$I$ is the identity matrix, under the operation
\[
(A, x) (B, y) = (AB, x + Ay).
\]
This group is denoted by~$\motion(n)$ and called the \defi{Euclidean
  motion group}.

We will now work with complex-valued functions over~$\motion(n)$.
There is also a definition of what it means for such a function to be
rapidly decreasing, though it is more technical than the definition
for functions over~$\R^n$.

A function~$f \in L^1(\motion(n))$ is of \defi{positive type} if
$f(A, x) = \overline{f((A, x)^{-1})}$ for all~$(A, x) \in \motion(n)$
and for all~$\rho \in L^\infty(\motion(n))$ we have
\[
\int_{\motion(n)} \int_{\motion(n)} f((B, y)^{-1} (A, x)) \rho(A, x)
\overline{\rho(B, y)}\, d(A, x) d(B, y) \geq 0.
\]
Here, we take as the measure the product of the Haar measure
for~$\so(n)$, normalized so that~$\so(n)$ has total measure~$1$, with
the Lebesgue measure for~$\R^n$.

\begin{theorem}
\label{thm:cong-pack}
Let~$\Kcal \subseteq \R^n$ be a convex body.  Suppose~$f\colon \motion(n)
\to \R$ is rapidly decreasing and that it satisfies the following conditions:

\begin{enumerate}
\item[(i)] $f(A, x) \leq 0$ whenever $(x + A \Kcal)^\circ \cap \Kcal^\circ
 = \emptyset$;

\item[(ii)] $\int_{\motion(n)} f(A, x)\, d(A, x) \geq \vol \Kcal$;

\item[(iii)] $f$ is of positive type.
\end{enumerate}

\noindent
Then the maximum density of a packing of congruent copies of~$\Kcal$
is at most~$f(I, 0)$.
\end{theorem}

The proof of Theorem~\ref{thm:cong-pack} is slightly more technical
than the proof of Theorem~\ref{thm:trans-pack}, but otherwise it
follows the same pattern.

Notice that the congruent packing graph is a Cayley graph whose vertex
set is the Euclidean motion group. So Theorem~\ref{thm:cong-pack} is
also an analogue of Theorem~\ref{thm:symmetryreduction}. It is,
however, more distant from Theorem~\ref{thm:symmetryreduction} than
Theorem~\ref{thm:trans-pack} is, since~$\R^n$ is Abelian,
but~$\motion(n)$ is not. This means that when one does harmonic
analysis over~$\motion(n)$, using the characters is not enough: one
also needs to consider higher-dimensional irreducible
representations, most of them are even infinite-dimensional. 

Though it is clear that Theorem~\ref{thm:cong-pack} can be restated in
the Fourier domain just like Theorem~\ref{thm:symmetryreduction}
could, it now becomes harder to carry out this procedure
explicitly~---~already for~$n = 2$ or~$3$, the formulas involved are
significantly more complicated than the ones for~$\R^n$. Using the
formulas effectively in a computational approach remains the main
obstacle in applying Theorem~\ref{thm:cong-pack}

\section{A computational approach}
\label{sec:explicit computations}

Theorem~\ref{thm:trans-pack} and Theorem~\ref{thm:cong-pack} might be
mathematically pleasing \textit{per se} but the real challenge is to
determine explicit functions giving good bounds. So far this has been
done only for a few cases. When~$N = 1$, Theorem~\ref{thm:trans-pack}
becomes a theorem of Cohn and Elkies~\cite{Cohn2003a}. The Cohn-Elkies
bound provides the basic framework for proving the best known upper
bounds for the maximum density of sphere packings in dimensions~$4$,
\dots,~$36$. It is also conjectured to provide tight bounds in
dimensions~$8$ and~$24$ and there is strong numerical evidence to
support this conjecture. De Laat, Oliveira, and
Vallentin~\cite{Laat2014a} have proposed a strengthening of the
Cohn-Elkies bound and computed better upper bounds for the maximum
density of sphere packings in dimensions~$4$, $5$, $6$, $7$, and~$9$.

Here we want to give an idea of how to set up a semidefinite program
for finding good functions. Let~$B_n$ be the unit ball in~$\R^n$. To
find bounds for the density of a sphere-packing, we want to find a
function $f\colon \mathbb{R}^n \to \mathbb{R}$ with $f(0)$ as small as
possible such that

\begin{enumerate}
\item[(i)] $f(x) \leq 0$ whenever $(x + B_n)^\circ \cap B_n^\circ =
  \emptyset$;

\item[(ii)] $\hat{f}(0) - \vol B_n \geq 0$;

\item[(iii)] $f$ is of positive type, which means that $\hat{f}(u)$ is
  nonnegative for all $u \in \mathbb{R}^n$.
\end{enumerate}

Without loss of generality we can assume that the function $f$ is even
and radial, i.e., $f(x)$ depends only on the norm of~$x$, so it is
essentially an even univariate function.  Another good thing is that
the Fourier transform of a radial function is radial again. Functions
whose Fourier transform have the form
\[
\hat{f}(u) = p(\|u\|) e^{-\pi \|u\|^2}, \quad \text{where $p$ is an
  even and univariate polynomial,}
\]
are dense in the space of rapidly decreasing even and radial
functions. Then by the Fourier inversion formula we can compute $f$
explicitly, monomial by monomial, through
\[
\int_{\R^n} \|u\|^{2k} e^{-\pi\|u\|^2} e^{2\pi i u \cdot x}\, du =
k!\, \pi^{-k} e^{-\pi \|x\|^2} L_k^{n/2-1}(\pi \|x\|^2),
\]
where~$L_k^{n/2-1}$ is the Laguerre polynomial of degree~$k$ with
parameter~$n/2 - 1$. These are orthogonal polynomials on the half open
interval $[0,\infty)$ with respect to the measure $x^{n/2-1} e^{-x}\,
dx$.

We specify function~$f$ via the polynomial~$p$. To do so, we fix~$d > 0$
and work with polynomials of degree up to~$2d$, that is, we work with
polynomials of the form
\[
p(t) = \sum_{k=0}^d a_{2k} t^{2k}.
\]
Working with finite~$d$ is our way of discretizing the Fourier domain,
a necessary step as we observed in Section~\ref{sec:harmonic analysis}.

Now constraints on~$f$ become constraints on~$p$, which can be
modeled as sum-of-squares constraints (see, e.g., the expository
papers of Lasserre and Parrilo in \textit{SIAG/OPT
  Views-and-News}~15~(2004)).  So we can set up a semidefinite
programming problem to find a function~$f$ satisfying the required
constraints:
\[
\begin{array}{rll}
\min&\sum_{k=0}^d a_k k!\, \pi^{-k} L_k^{n/2-1}(0)\\[4pt]
&p(t) = \sum_{k=0}^d a_{2k} t^{2k},\\[4pt]
&\sum_{k=0}^d a_k k!\, \pi^{-k} L_k^{n/2-1}(\pi w^2)\\[4pt]
&\qquad{} + v_d^\tp(w) R
v_d(w) + (w^2 - 2^2) v_{d-2}^\tp(w) S v_{d-2}(w) = 0,\\[4pt]
&p(0) - \vol B_n \geq 0,\\[4pt]
&p(t) = v_d^\tp(t) Q v_d(t),\\[4pt]
&\text{$Q$, $R$, $S$ are positive semidefinite matrices,}
\end{array}
\]
where $v_d(z) = (1, z, \ldots, z^d)$ is the vector of all monomials up
to degree~$d$.

From a numerical perspective this formulation is a
catastrophe --- a fact well known to specialists in the field --- since
the monomial basis is used.  Even though the resulting semidefinite
program is small, say when we use $d = 10$, it is impossible to get a
solution from standard numerical solvers.  On the other hand there
are many equivalent ways to implement this program by using different
choices of polynomial bases. Here we have two choices: one for the
vectors~$v_d$ and one for testing the polynomial identities. With
quite some experimentation we found that the basis
\[
P_k(t) = \mu_k^{-1} L_k^{n/2-1}(2 \pi t),
\]
where~$\mu_k$ is the absolute value of the coefficient
of~$L_k^{n/2-1}(2 \pi t)$ with largest absolute value, performs well.

We believe that the problem of finding a good basis deserves further
investigation. Currently almost nothing (to the best of our knowledge
only the papers by Löfberg and Parrilo~\cite{Loefberg2004a} and Roh
and Vandenberghe~\cite{Roh2006a} address this issue) is known about it
although it is a crucial factor for solving polynomial optimization
problems in practice.

Another use of Theorem~\ref{thm:trans-pack} is to provide bounds for
binary sphere packings. These are packings of balls of two different
sizes, i.e., we have~$N = 2$ and~$\Kcal_1$, $\Kcal_2$ are
balls. Binary sphere packings occur naturally in applications such as
materials science and chemistry. De Laat, Oliveira, and
Vallentin~\cite{Laat2014a} used Theorem~\ref{thm:trans-pack} to
compute upper bounds for the maximum densities of binary sphere
packings in dimensions~$2$, \dots,~$5$.

Recently, Oliveira and Vallentin~\cite{Oliveira2013a} used
Theorem~\ref{thm:cong-pack} to compute upper bounds for the densities
of pentagon packings. Here yet a new challenge arises: The Fourier
transform is no longer matrix-valued but takes infinite-dimensional
Hilbert-Schmidt kernels as values. Oliveira and Vallentin determined a
first upper bound ($0.98$ in comparison to the best known lower bound
of $0.92$) and the numerical result obtained gives hope that the
theorem will also be useful in the case of tetrahedra packings to meet
the challenge of Gravel, Elser and Kallus.

\section{Conclusion}

It is natural to consider optimization methods when dealing with
geometric packing problems, and we tried to show how well-known
methods from combinatorial optimization, namely the Lovász theta
number and its variants, can be extended so as to provide upper
bounds for the packing density. Such extensions provide a uniform
framework to deal with geometric packing problems.

For finite graphs, only in very specific cases does the Lovász theta
number provide tight bounds. The same happens for geometric packing
graphs: only in very few cases are the bounds coming from extensions
of the theta number tight; in most cases, such bounds are but a first
step in solving the problem.

The link made with combinatorial optimization techniques not only
allows us to provide a unified framework and to have access to
well-known optimization tools, it also points out to ways in which
such bounds can be strengthened. The obvious approach is to extend ideas
like the Lasserre hierarchy to geometric packing problems. Such
higher-order bounds can incorporate more sophisticated constraints
like those coming from the local interaction of more than two
vertices; in other words, we then deal with $k$-point correlation
functions and not only with $2$-point correlation functions.

Schrijver~\cite{Schrijver2005a} considered $3$-point correlation
functions for binary codes and Bachoc and Vallentin~\cite{Bachoc2008a}
used $3$-point correlation functions for packings of spherical caps
on the unit sphere. De Laat and Vallentin~\cite{Laat2013a} recently
showed that this approach has the (theoretical) potential to solve all
geometric packing problems. However, the price to pay is that the size
of the optimization problems involved grows fast.

The success of such techniques will depend on several factors, among
which: (i)~how to analyze the optimization problem without using a
computer, for instance to find asymptotic results, (ii)~how to
automatize the use of harmonic analysis, and (iii)~how to solve
semidefinite programs involving sums-of-squares constraints in an
efficient and numerically stable manner.

\end{document}